\documentclass[11pt]{amsart} 

\usepackage{amsmath}
\usepackage{amssymb}
\usepackage{graphicx}

\newtheorem{theorem}{Theorem}[section]

\theoremstyle{definition}
\newtheorem{definition}[theorem]{Definition}

\numberwithin{equation}{section}

\title[Optimal time averages]{Optimal time averages in non-autonomous nonlinear dynamical systems}

\author[C. R. Doering]{Charles R. Doering}

\address[C. R. Doering]{University of Michigan, Ann Arbor MI, USA}
\email{{\tt doering@umich.edu}}

\author[A. McMillan]{Andrew McMillan}
\address[A. McMillan]{University of Michigan, Ann Arbor MI, USA}
\email{\tt andrewnm@umich.edu}

\keywords{dynamical systems, ergodic optimization, polynomial optimization, semidefinite programming, 
sum-of-squares optimization}

\subjclass[2010]{37C99,49N99,70K99}

\begin{document}

\begin{abstract}
The auxiliary function method allows computation of extremal long-time averages of functions of dynamical variables in autonomous nonlinear ordinary differential equations via convex optimization.
For dynamical systems defined by autonomous polynomial vector fields, it is operationally realized as a semidefinite program utilizing sum of squares technology. 
In this contribution we review the method and extend it for application to periodically driven non-autonomous nonlinear vector fields involving trigonometric functions of the dynamical variables.
The damped driven Duffing oscillator and periodically driven pendulum are presented as examples to illustrate the auxiliary function method's utility.
\end{abstract}

\maketitle


\section{Introduction}

Dynamical systems governed by ordinary differential equations (ODEs) can have complex global attractors containing complicated and chaotic solutions.
The primary interest in such cases is often on statistics of solutions, e.g., long-time averages of functions of the dynamical variables.
Averages along trajectories generally depend on initial conditions and it is natural to seek the largest or smallest such averages among all solutions, as well as the extremal trajectories that realize them.
Moreover, for various purposes---including, notably, “control of chaos” \cite{OGY1990,HO1996,YangT00}---it is valuable to know extremal trajectories regardless of their stability. 

The brute-force approach to searching for extremal time averages is to construct a large number of candidate trajectories which is ad hoc, computationally expensive, and operationally limited to sufficiently stable solutions.
An alternative approach that is broadly applicable and often more tractable is to construct sharp a priori bounds on long-time averages via convex optimization.
In this paper we review a mathematical device that has come to be known as the {\em auxiliary function method} \cite{Chern14, Fantuzzi16,Tobasco18} and its computational implementation, and describe some new developments to generalize it applicability.
To illustrate the method's utility we apply the tools to the damped driven Duffing oscillator and periodically driven pendulum.

To introduce the auxiliary function method we focus on determining upper bounds for time averages of functions of the dynamical variable for autonomous ODEs; lower bounds are analogous.
Consider $\textbf{x}(t)\in \mathbb{R}^d$ satisfying
\begin{equation}
    \frac{d\textbf{x}}{dt}=\textbf{f}(\textbf{x})
    \label{AODE}
\end{equation}
for continuously differentiable vector fields $\textbf{f}(\textbf{x})$.
When there is no confusion, we will denote the vector components of $\textbf{x}(t)$ and $\textbf{f}(\textbf{x})$ as ${x}_i(t)$ and $f_i(\textbf{x})$, respectively.

Given a quantity of interest $\Phi(\textbf{x})$, define its long-time average along the trajectory $\textbf{x}(t)$ with $\textbf{x}(0) = \textbf{x}_0$ by
\begin{equation}
    \overline{\Phi}(\textbf{x}_0)=\limsup\limits_{T \rightarrow \infty}\frac{1}{T}\int_0^T\Phi (\textbf{x}(t))dt.
\end{equation}
The choice of $\Phi(\textbf{x})$ is subject to the particular application in mind.
Let $\text{B} \subset \mathbb{R}^d$ be a compact invariant region in the phase space.
In a dissipative system $\text{B}$ could be an absorbing compact set, or in a conservative system $\text{B}$ could be defined by constraints on dynamical invariants.
We are interested in the maximal long-time average among all trajectories (eventually) remaining in $\text{B}$, i.e.,
\begin{equation}
  \overline{\Phi}^*=\max_{\textbf{x}_0\in \text{B}}\overline{\Phi}(\textbf{x}_0).
  \label{phibarstar}
\end{equation}
The fundamental questions are: what is the value of $\overline{\Phi}^*$ and what trajectories attain it?

Upper bounds on averages can be deduced using the fact that time derivatives of bounded functions average to zero.
This elementary observation follows from the fact that for every $\text{V}(\textbf{x}) \in C^1(\text{B})$---the set of continuously differentiable functions on B---we have
\begin{equation}
    0=\limsup\limits_{T \rightarrow +\infty}\frac{\text{V}(\textbf{x}(T)-\text{V}(\textbf{x}(0))}{T}=\overline{\frac{d}{dt}\text{V}(\textbf{x}(\cdot))}=\overline{\textbf{f}(\textbf{x}(\cdot))\cdot\nabla \text{V}(\textbf{x}(\cdot))}.
\end{equation}
We hereafter refer such $\text{V}(\textbf{x}) \in C^1(\text{B})$ as ``auxiliary'' functions.
Note that equation (4) holds for any auxiliary function so there is an infinite family of functions with the same time average as $\Phi(\textbf{x})$.  In particular,
\begin{equation}
    \overline{\Phi}(\textbf{x}_0)=\overline{ [ \Phi+\textbf{f} \cdot\nabla \text{V} ] }(\textbf{x}_0).
\end{equation}

For any auxiliary function one can obtain a trivial upper-bound on $\overline{\Phi}(\textbf{x}_0)$ by bounding the right hand-side point-wise on $\text{B}$ and susequently maximizing the left hand side over initial data $\textbf{x}_0$:
\begin{equation}
     \overline{{\Phi}}^* \leq \max_{\textbf{x}\in \text{B}}[\Phi(\textbf{x})+\textbf{f}(\textbf{x})\cdot\nabla \text{V}(\textbf{x})].
     \end{equation}
The best such {\it a priori} upper bound on $\overline{{\Phi}}^*$ is then
\begin{equation}
    \overline{{\Phi}}^*\leq \inf_{\text{V}\in C^1(\text{B})}\max_{\textbf{x}\in \text{B}} \ [\Phi(\textbf{x})+\textbf{f}(\textbf{x})\cdot \nabla \text{V}(\textbf{x})].
    \label{ineq}
\end{equation}

The minimization over the right hand side of (7) is a {\it convex} optimization in the auxiliary function V.
Indeed, define the functional
\begin{equation}
    \mathcal{F}(\text{V})= \max_{\textbf{x}\in \text{B}} \ [\Phi(\textbf{x})+\textbf{f}(\textbf{x})\cdot \nabla \text{V}(\textbf{x}))]
\end{equation}
and insert a convex combination of auxiliary functions and apply the triangle inequality to deduce
 \begin{eqnarray}
     &&\mathcal{F}(\lambda \text{V}_1+(1-\lambda)\text{V}_2) = \max_{\textbf{x}\in \text{B}}[\Phi(\textbf{x})+\textbf{f}(\textbf{x})\cdot \nabla(\lambda \text{V}_1(\textbf{x})+(1-\lambda)\text{V}_2(\textbf{x}))] \nonumber \\
     &&= \ \ \max_{\textbf{x}\in \text{B}} \ [\lambda\{\Phi(\textbf{x})+\textbf{f}(\textbf{x})\cdot \nabla \text{V}_1(\textbf{x})\}+(1-\lambda))\{\Phi(\textbf{x}) + \textbf{f}(\textbf{x})\cdot\nabla\text{V}_2(\textbf{x}))\}] \nonumber \\
    &&\leq \ \lambda \max_{\textbf{x}\in \text{B}} \ [\Phi(\textbf{x})+\textbf{f}(\textbf{x})\cdot \nabla \text{V}_1(\textbf{x})] + (1-\lambda)\max_{\textbf{x}\in \text{B}}[\Phi(\textbf{x})+\textbf{f}(\textbf{x})\cdot \nabla \text{V}_2(\textbf{x})] \nonumber \\
     &&= \lambda \mathcal{F}(\text{V}_1)+(1-\lambda)\mathcal{F}(\text{V}_2). \nonumber 
\end{eqnarray}
\indent The remarkable fact is that the inequality in (\ref{ineq}) is actually an {\it equality}:
\begin{equation}
    \overline{\Phi}^* = \inf_{\text{V}\in C^1(\text{B})} \max_{\textbf{x}\in \text{B}} [\Phi(\textbf{x})+\textbf{f}(\textbf{x})\cdot \nabla \text{V}(\textbf{x})].
    \label{eq}
\end{equation}
Details of the proof can be found elsewhere \cite{Tobasco18} but we sketch it here in four lines for completeness:
\begin{equation*}
    \begin{split}
        \max_{\textbf{x}_0\in \text{B}}\overline{\Phi}=\max_{\substack{\mu\in \text{Pr(B)}\\ \mu\, \text{is invar.}}}\int \Phi d\mu\\
        =\sup_{\mu \in \text{Pr(B)}}\inf_{\text{V}\in C^1(\text{B})}\int (\Phi +f\cdot \nabla \text{V}) d\mu\\
        =\inf_{\text{V}\in C^1(\text{B})}\sup_{\mu \in \text{Pr(B)}}\int (\Phi +f\cdot \nabla \text{V} ) d\mu\\
        =\inf_{\text{V}\in C^1(\text{B})} \max_{\textbf{x}\in \text{B}} [\Phi(\textbf{x})+\textbf{f}(\textbf{x})\cdot \nabla \text{V}(\textbf{x})]
    \end{split}
\end{equation*}
The key observations above are (i) that time averages can be realized as phase space averages against invariant measures, (ii) maximizing over invariant probability measures can be realized as a Lagrange multiplier problem where $\int f\cdot \nabla \text{V}  \, d\mu = 0 $ for all $\text{V}$ ensures $\mu$ is invariant, (iii) swapping the order of supremum and infimum can be performed due to standard abstract min-max theorems, and (iv) the $\sup_{\mu \in \text{Pr(B)}} \int (\Phi +f\cdot \nabla \text{V} ) d\mu$ is realized by a delta-mass located where $\Phi(\textbf{x})+\textbf{f}(\textbf{x})\cdot \nabla \text{V}(\textbf{x})$ assumes its maximum.

Thus arbitrarily sharp bounds on the maximal long-time average are available via convex optimization over auxiliary functions. Optimal or sequences of near-optimal V produce optimal or sequences of increasingly near-optimal bounds. 
Moreover, if $\text{V}\in C^1(\text{B})$ is an optimal auxiliary function, then it's straightforward to see that the corresponding optimal trajectory or trajectories reside in the subset of B where the continuous function $\Phi(\textbf{x})+\textbf{f}(\textbf{x})\cdot \nabla \text{V}(\textbf{x}) = \overline{\Phi}^*$.
Likewise if V is just near-optimal, then corresponding near-optimal trajectories spend a significant fraction of time in “high altitude” level sets of $\Phi(\textbf{x})+\textbf{f}(\textbf{x})\cdot \nabla \text{V}(\textbf{x})$.
Either way the auxiliary function approach can be used to localize extremal trajectories in the phase space; see \cite{Tobasco18} for further details and an example application to the Lorenz equations.

On the surface the minimization over auxiliary functions in (\ref{eq}) seems computationally intractable, as the optimization must be performed over an infinite dimensional function space.
However, there are two key observations to be made.
The first is that (\ref{phibarstar}) is equivalent to finding
\begin{equation}
    \begin{gathered}
        \min\, \text{U} \\
        \text{s.t.} \,\, \overline{\Phi (\textbf{x})}\leq \text{U}\, ,\,\forall \textbf{x}\in \text{B}.
        \label{11}
    \end{gathered}
\end{equation}
and a sufficient condition for the constraint is that $\Phi(\textbf{x})+\textbf{f}(\textbf{x})\cdot \nabla V (\textbf{x})\leq \text{U}$ for all $\textbf{x}\in B$ because a global point-wise constraint obviously implies a global average constraint.
Therefore, (\ref{11}) can be replaced with a minimization subject to a point-wise non-negative constraint,
\begin{equation}
\begin{gathered}
\min \text{U}\\
   \text{s.t.}\,\, \text{U}-\Phi(\textbf{x})-\textbf{f}(\textbf{x})\cdot \nabla \text{V} (\textbf{x})\geq 0\,,\,\forall \textbf{x}\in \text{B}.
    \end{gathered}
    \label{12}
\end{equation}
As stated, the problem reduces to determining the non-negativity of a given multivariate function.
Unfortunately determining the non-negativity of multivariate functions is NP hard \cite{NP}, but in section 2, we will formulate the problem as a semidefinite program and perform suitable and somewhat natural relaxations to make the problem computationally accessible. 

Meanwhile it is important to recognize that interest in extreme time averages is not new.
In the abstract dynamical systems community it goes under the name ``ergodic optimization'' \cite{OJ2006,B2018} and was  motivated in part by conjectures late last century that many quantities of interest in applications for chaotic dynamical systems are optimized, in a time averaged sense, on (relatively) simple unstable periodic orbits \cite{HO1996}.
Those conjectures, in turn, underly ``control of chaos'' notions \cite{OGY1990} that emerged earlier.

Rather than developing theoretical or quantitative computational tools to evaluate extreme time averages, however, the ergodic optimization field focused on more conceptual questions resulting in theorems such as that every ergodic measure is the unique maximizing measure for some continuous function.
(In our setting this is the statement that for every initial condition $\textbf{x}_0 \in \text{B}$, $\overline{\Phi}(\textbf{x}_0) = \overline{\Phi}^*$ for some continuous function $\Phi$.)
The ergodic optimization community recognized the variational structure reflected in (\ref{eq}) and---given complete knowledge of the flow map\footnote{The flow map takes $\textbf{x}_0$ to $\textbf{x}(t)$ along the ensuing trajectory for each time $t > 0$.}---proposed a strategy to produce a sequence of increasingly near-optimal auxiliary functions \cite{CG1993}.
In section 3 we offer an elementary example to explicate both the ideas discussed in this section and those developed in the immediately following section 2.

\section{Semidefinite Programming, Sum of Squares Technology, and Polynomial Dynamics }
\subsection{Semidefinite Programming}
Computing upper and lower bounds on the quantity of interest, $\overline{\Phi}$, can be simplified to a convex optimization problem over a finite dimensional vector space of polynomials in a Semidefinite Program (SDP) under suitable relaxations.
In general, a SDP takes $C, A_i\in \mathbb{R}^{n\times n}$ and $b\in \mathbb{R}^m$ for $i\in \{1,2,..,m\}$ as inputs with the goal of determining
\begin{equation}
    \begin{gathered}
         \min_{X\in \mathbb{R}^{n\times n}} \langle C,X \rangle\\
         \text{s.t.}\, \, \langle A_i, X\rangle = b_i\, \text{for $i\in \{1,2,.., m\}$ and} \ X\succeq 0
    \end{gathered}
    \label{13}
\end{equation}
where for two matrices $B,D\in \mathbb{R}^{n\times n}$, $\langle B , D \rangle=\sum_{i=1}^n\sum_{j=1}^nb_{i,j}d_{i,j}$ and $X\succeq 0$ means $X$ is positive semi-definite.
Equation (\ref{13}) is frequently called the \textit{primal} problem of a SDP which also has a \textit{dual} problem of the form
\begin{equation}
\begin{split}
    \max_{y\in \mathbb{R}^m} \langle b, y\rangle\\
     \text{s.t.}\, \, C \succeq \sum_{i=1}^m y_i A_i
    \end{split}
    \label{14}
\end{equation}
where $P\succeq Q$ means $P-Q\succeq 0$.
If the solutions to both the primal and dual problem are the same, then we say that the SDP has a \textit{dual gap} of zero.
See \cite{Boyd} for a review of SDPs and their applications.

An important class of SDPs are polynomial optimization problems. In particular, one is frequently interested in optimizing a multivariate polynomial subject to a set of non-negative constraints. If the polynomial to optimize is given to be $p(\textbf{x})$ such that $g_{i}(\textbf{x})\geq 0$ for $i \in \{1,2,..,m\}$ and $\textbf{x}\in \mathbb{R}^d$, the  problem is of the form \begin{equation}
    \begin{gathered}
    \min_{\textbf{x}\in \mathbb{R}^d} p(\textbf{x})\\
    \text{s.t.}\,\, g_{i}(\textbf{x})\geq 0\,\text{for $i \in \{1,2,..,m\}$}.
    \label{above}
    \end{gathered}
\end{equation}
In the next Section 2.2 we will see how problems of the form (\ref{above}) may written as a SDP just as in (\ref{13}). 
SDPs are now somewhat standard and easily implementable as there are a wide array of various soft-wares to solve well-posed problems of the form (\ref{13}) and (\ref{14}).
One needs only a parser, such as YALMIP, SOSTOOLS, or GloptiPoly, and a semi-definite program solver, such as Mosek, SeDumi, or SCS.
Many of these are implementable in a standard Mathlab toolbox. 
Computations in this paper were performed using Yalmip paired with Mosek.

\subsection{Positivity of Polynomials}
\subsubsection{Global Positivity}
Differential equations for many applications are purely polynomial in their arguments.
That is, one is frequently interested in $\dot{\textbf{x}} = \textbf{f}(\textbf{x})$ for $f_i(\textbf{x})\in \mathbb{R}[\textbf{x}] $, where $\mathbb{R}[\textbf{x}]$ is the vector space of all polynomials over $\textbf{x}$.
If $f_i(\textbf{x})$ is polynomial and $\text{V}(\textbf{x})$ is restricted to $\mathbb{R}[\textbf{x}]$, then the constraint in (\ref{12}) simplifies to determining whether a multi-variate polynomial is non-negative.
Even determining the non-negativity of a multivariate polynomial is \textit{still} NP hard, however, except for an extremely limited set of examples such as uni-variate or quadratic polynomials \cite{NP}.
The key observation is that determining the stronger condition that the polynomial is a Sum of Squares (SOS) of polynomials is a problem can be solved in polynomial time \cite{parillo1,parillo2}.
\begin{definition} A polynomial $\text{p}(\textbf{x})\in \mathbb{R}[\textbf{x}]$ is a \textbf{Sum of Squares} if there is a finite collection of polynomials $p_i(\textbf{x})\in \mathbb{R}[\textbf{x}]$ such that $\text{p}(\textbf{x})=\sum_{i=1}^N {[{p_i}(\textbf{x})]}^2$.
We denote $\mathcal{S}[\textbf{x}]$ and $\mathcal{S}_d[\textbf{x}]$ as the cones of all SOS polynomials and all SOS polynomials up to degree d, respectively.
\end{definition}
\noindent
This sum of squares condition is of course sufficient for the non-negativity of a polynomial, and it is necessary if the polynomials in question are up to quadratic \cite{Hilbert}, but in general being SOS is not equivalent to non-negativity.
Therefore, one might be concerned that the proposed strengthening, of going from a non-negative polynomial to one with a SOS representation, has given up too much.
However, there is a wonderful result \cite{Lassere}, which states that SOS polynomials are dense in the space of non-negative, real polynomials of arbitrary degree and of arbitrary dimension in the $\ell_1$ norm of the polynomial's coefficients.
There is quite a rich history in determining whether a polynomial, or more generally a rational function, can be written as a SOS or a sum of rational functions with square numerators and denominators dating back to Hilbert's 17th problem; see \cite{Marshall} for a historical review. 

Fortunately for applications there is a simple yet computationally useful result about representations of SOS polynomials:
\begin{theorem} Given a multi-variate polynomial $p(\textbf{x})$ in n variables and of degree 2d, $p(\textbf{x})$ is representable as a sum of squares if and only if there exists a positive semi-definite  and symmetric matrix Q such that 
\begin{equation*}
    p(\textbf{x})=z(\textbf{x})^TQz(\textbf{x}),
\end{equation*}
where  $z(\textbf{x})=[1, x_1,x_2,..,x_n,x_1x_2,..,x_n^d]$.
\end{theorem}
\begin{proof}
The ``if'' is evident.
Conversely, suppose that $p(\textbf{x})$ has a sum of squares representation.
Then
\begin{equation*}
\begin{split}
    p(\textbf{x})=\sum_{i=1}^n{q_{i}(\textbf{x})}^2
    =\sum_{i=1}^n {({a_i}^Tz(\textbf{x}))}^2\\
    =\sum_{i=1}^n (z^T(\textbf{x})a_i)({a_i}^Tz(\textbf{x}))\\
    =z^T(\textbf{x})\big (\sum_{i=1}^na_i{a_i}^T\big )z(\textbf{x})
    =z^T(\textbf{x})Qz(\textbf{x})
    \end{split}
\end{equation*}
\end{proof}
\noindent
Therefore, determining whether an even degree, non-negative polynomial is a SOS is equivalent to finding a positive semi-definite and symmetric matrix, Q, such that 
\begin{equation}
    p(\textbf{x})=z(\textbf{x})^TQz(\textbf{x})\geq 0,
    \label{16}
\end{equation}
where $z(\textbf{x})$ is a suitable polynomial basis. 
We compute an example for demonstration purposes: suppose we wish to represent $f(x,y)=2x^4+5y^4+x^2y^2$, a SOS polynomial, in the form of (\ref{16}). 
Write
\begin{equation}
    \begin{split}
        f(x,y)=2x^4+5y^4+x^2y^2 
        =[x^2\, y^2\, xy]^T
        \begin{bmatrix}
q_{11} & q_{12} & q_{13}\\
q_{12} & q_{22} & q_{23}\\
q_{13} & q_{23} & q_{33}\\
\end{bmatrix}[x^2\, y^2\, xy]\\
=q_{11}x^4+q_{22}y^4+(q_{33}+2q_{12})x^2y^2+2q_{13}x^3y+2q_{23}xy^3.
    \end{split}
\end{equation}
Equating coefficients we find
\begin{equation}
q_{11}=2\,, q_{22}=5 ,\, q_{33}+2q_{12}=1, \,q_{23}=0\,,  \,q_{13}=0
\end{equation}
so that the matrix is positive semi-definite for $-\sqrt{10} \le q_{12} \le \frac{1}{2}$ with $q_{33}=1-2q_{12}$.

\subsubsection{Local Positivity}

The SOS criterion in (\ref{16}) is a global condition insofar as it insists that our desired polynomial is non-negative for all $\textbf{x}\in \mathbb{R}^d$.
But frequently one is satisfied with  local positivity of a polynomial, and this was the original formulation of the problem in (\ref{12}).
Due to the robustness of characterizing regions in phase space with polynomials, we can restrict our attention to locality constraints defined in terms of only polynomials.
\begin{definition} A set $\textbf{K}$ is called \textbf{semi-algebraic} if $\textbf{K}$ is defined by finitely many polynomial equalities or inequalities. A prototypical example of such \textbf{K} is
\begin{equation}
    \textbf{K}:=\{\textbf{x}\in \mathbb{R}^d\, |\, g_i(\textbf{x})\leq 0\, , h_j(\textbf{x})=0\, \, \text{for} \,i=1,...,m \,\text{and} \, j=1,...,n\},
    \label{19}
\end{equation}
where $g_i(\textbf{x}), h_j(\textbf{x}) \in \mathbb{R}[\textbf{x}]$.
\bigskip

\noindent
We now look to enforce equation (\ref{16}) under the strengthened constraint
\begin{equation}
    p(\textbf{x})=z(\textbf{x})^T Qz(\textbf{x})\geq 0\,, \, \forall \textbf{x}\in \text{semi-algebraic set {\bf K}}.
    \label{20}
\end{equation}
\end{definition}
One way of viewing the localized constraint of being within \textbf{K} is to say that for all $\textbf{x}\in \textbf{K}$
\begin{equation}
\begin{split}
    \sum_{j=1}^nh_j(\textbf{x})r_j(\textbf{x})=0\,,\,\forall r_j(\textbf{x})\in \mathbb{R}[\textbf{x}] \quad \text{and}\\
    \sum_{i=1}^mg_i(\textbf{x})s_i(\textbf{x})\leq 0\,, \, \forall s_i(\textbf{x})\in \mathbb{R}^+[\textbf{x}],
    \end{split}
\end{equation}
where $\mathbb{R}^+[\textbf{x}]$ is the positive cone in $\mathbb{R}[\textbf{x}]$. 
This can be realized by writing (\ref{20}) as 
\begin{equation}
    \begin{gathered}
        \text{Find} \,r_1,..,r_n\,\text{and}\, s_1,..,s_m\\
        \text{s.t.}\,\, p(\textbf{x})+ \sum_{i=1}^nh_i(\textbf{x})r_i(\textbf{x})+\sum_{j=1}^mg_j(\textbf{x})s_j(\textbf{x})\geq 0\\
        s_1,..,s_m\in \mathcal{S}[\textbf{x}]
    \end{gathered}
    \label{22}
    \end{equation}
 where we've replaced the positive cone condition with being representable as a SOS. Equation (\ref{22}) is frequently called the \textit{S-Procedure}, where the ``S'' comes from the SOS constraints on the $s_i$ polynomials.

\subsubsection{Sum of Squares in Dynamical Systems}
Returning to (\ref{12}), replacing positivity of $\text{U}-\Phi(\textbf{x})-\textbf{f}(\textbf{x})\cdot \nabla \text{V}$ on all of ${\mathbb R}^d$ to SOS allows us to relax the problem to 
\begin{equation}
    \begin{gathered}
        \min\, \text{U} \\
        \text{s.t.} \,\text{U}-\Phi(\textbf{x})-\textbf{f}(\textbf{x})\cdot \nabla \text{V} (\textbf{x})\in \mathcal{S}[\textbf{x}].
    \end{gathered}
    \label{23}
\end{equation}
Moreover, if we wish to ensure that the polynomial in question is only locally positive on the compact set $\text{B} \subset$ semi-algebraic set \textbf{K} defined as in (\ref{19}), we can augment equation (\ref{23}) with
\begin{equation}
\begin{gathered}
         \min\, \text{U} \\
        \text{s.t.} \,\text{U}-\Phi(\textbf{x})-\textbf{f}(\textbf{x})\cdot \nabla \text{V} (\textbf{x}) + \dots \\
        \dots +\sum_{i=1}^m h_i(\textbf{x}) r_i(\textbf{x})+\sum_{j=1}^n g_j(\textbf{x}) s_j(\textbf{x}) \in \mathcal{S}[\textbf{x}] \   \text{and}\\
        s_1(\textbf{x}),...,s_n(\textbf{x})\in \mathcal{S}[\textbf{x}].
\end{gathered}
\end{equation}

A few key remarks are required here.
The polynomials $f_i(\textbf{x})$ are exogenously given as part of the dynamical system in question but there are choices to be made for polynomials $\Phi (\textbf{x})$ and $\text{V}(\textbf{x})$.
$\Phi(\textbf{x})$ is chosen according to the particular application in mind.
However, upon further inspecting the programs from a computational perspective, it turns out that the resulting U is generally very sensitive to the choice of $\text{V}(\textbf{x})$.
In particular the degree of $\text{V}(\textbf{x})$ is pertinent, and the reason is two fold.

Firstly, if the degree of $\text{V}(\textbf{x})$ is too small then the SOS constraint may fail to be feasible even within a reasonable tolerance for numerical error.
Secondly, the resulting U may fail to be a sharp upper bound for $\overline{\Phi}$.
The restriction that $\text{V}(\textbf{x})$ is polynomial is completely absent in (\ref{11}) as well as (\ref{12}), the original problem and its slight strengthening, so it is unreasonable to expect that sharp bounds can be achieved by restricting to the space of polynomials.
Fortunately, though, there is the following result \cite{Lakshmi}:
\begin{theorem} Suppose that $\textbf{K}$ is a compact, semi-algebraic set defined in terms of $\{g_i\}_{i=1}^m$. Let $s=\max_i \text{deg}(g_i)$, $r=\text{deg}(\text{U}-\Phi-\textbf{f}\cdot \nabla \text{V})$ and $\Gamma_d$ denote the set of polynomials that are a weighted sum of the $g_i$'s, where the weights are SOS polynomials of degree no more than $r-s$. If there exists $L$ such that $L-{||\textbf{x}||}^2\in \Gamma_d$ for some $d$, then
\begin{equation*}
    \overline{\Phi}^*=\lim\limits_{d\rightarrow \infty}\inf_{\substack{\text{U}\in \mathbb{R} \\ V\in \mathbb{R}_d[\textbf{x}]}}\{\text{U} \, |\, \text{U}-\Phi(\textbf{x})-\textbf{f}(\textbf{x})\cdot \nabla \text{V} (\textbf{x})\in \Gamma_d\}.
\end{equation*}
\end{theorem}

\smallskip
\noindent
Therefore, by taking the polynomial degree of our auxiliary function to infinity we are guaranteed to achieve the sharp bounds of the theoretical formalism in Section 1, i.e., in theory we have lost nothing in restricting $V$ to being polynomial.
Coupled with Lassere's density result \cite{Lassere} it is operationally reasonable to restrict our attention to polynomial auxiliary functions as well as positive polynomials with a sum of squares representation.

In practice, one incrementally increases the allowed degree of V and the bounds are declared sharp if increasing the degree only yields small (near numerical precision) improvements in the bounds U.
In practice sharp bounds may be achieved for auxiliary functions of relatively small degree---say, around degree 8 or 10---that are computationally accessible on a standard laptop for systems with relatively low degrees of freedom.  

\bigskip

\section{A Simple Example}
\indent To illustrate the ideas introduced above in the context of a concrete example, consider the one dimensional polynomial dynamical system
\begin{equation}
    \frac{dx}{dt} = x - x^3 = f(x)
    \label{example}
\end{equation}
and quantity of interest
\begin{equation}
    \Phi(x) = x^2.
\end{equation}
For our purposes (\ref{example}) possesses three classes of solutions corresponding to three classes of initial data:
\begin{eqnarray}
    \nonumber x(t) \rightarrow -1 &\text{for}& -\infty < x_0 < 0, \\
    x(t) \rightarrow 0  &\text{for}& x_0 = 0, \ \  \text{and} \\
    x(t) \rightarrow +1  &\ \text{for}& 0 > x_0 > \infty. \nonumber \label{ex}
\end{eqnarray}
Therefore $\overline{\Phi}(0)=0$ and $\overline{\Phi}(x_0)=1$ for all $x_0 \ne 0$ so that $\overline{\Phi}^* = 1$.  But how might one discern this within the auxiliary function formulation? 

In this example it is easy to divine an optimal polynomial auxiliary function, namely $V(x) = \frac{1}{2}x^2$.
Indeed, 
\begin{eqnarray}
    \Phi + f V' &=& x^2 + (x-x^3)x \nonumber  \\
   &=& 2 x^2 - x^4 \\
   &=& 1 - (x+1)^2 (x-1)^2. \nonumber 
\end{eqnarray}
That is, for this optimal auxiliary function $\Phi(x) + f(x) V'(x) = \overline{\Phi}^* - S(x)$ where $S(x)$ is a (sum of) square(s) of polynomials.

Moreover, for this particular $\Phi(x)$ and optimal auxiliary function $V(x)$, the quantity $\Phi(x) + f(x) V'(x)$ achieves its pointwise maximum $\overline{\Phi}^*$ precisely---and only---at $x=\pm1$, points that are both optimal initial conditions (but not uniquely so---any $x_0 \ne 0$ is optimal) and optimal trajectories such that every neighborhood thereof hosts {\it every} optimal trajectory 100\% of the time over the infinite time interval of averaging.

Given our quantitative analytical knowledge of the flow map for this simple example, however, by relaxing the polynomial restriction we can also conceive a sequence of near-optimal auxiliary functions $V_{\epsilon} \in C^1(\mathbb{R})$ so that that $\lim_{\epsilon \rightarrow 0} \{ \Phi(x) + f(x)V_{\epsilon}'(x) \} = \overline{\Phi}^*$ for {\it every} $x \ne 0$.
Indeed, for every $\epsilon > 0$ define
\begin{equation}
    V_{\epsilon}(x) = \frac{1}{2} \ln{(x^2+\epsilon)}
\end{equation}
so that
\begin{eqnarray}
    \Phi(x) + f(x) V_{\epsilon}'(x)  &=& x^2 + \frac{(x-x^3)x}{x^2+\epsilon} \nonumber  \\
   &=& \frac{(1+\epsilon)x^2}{x^2+\epsilon}.
\end{eqnarray}
Then it is evident that $V_{\epsilon}$ is an increasingly near-optimal sequence of auxiliary functions in the sense that $\inf_{\epsilon > 0} \, \sup_x \{\Phi(x) + f(x)V_{\epsilon}'(x)\} = \overline{\Phi}^*$ and furthermore, as advertised, $\lim_{\epsilon \rightarrow 0} \{ \Phi(x) + f(x)V_{\epsilon}'(x) \} = \overline{\Phi}^*$ for every $x \ne 0$.
Note as well that $\Phi(x) + f(x) \times [\lim_{\epsilon \rightarrow 0} V'_{\epsilon}(x)] = \overline{\Phi}^*$ for every $x \ne 0$ even though the limit of the sequence $V_{\epsilon}$, i.e., $\ln{|x|}$, is not $C^1$.

But even {\it more} is true about this sequence of auxiliary functions: for every $x \in \mathbb{R}$
\begin{equation}
   \lim_{\epsilon \rightarrow 0} \{\Phi(x) + f(x)V_{\epsilon}'(x)\} = \overline{\Phi}(x).
\end{equation}
That is, $\Phi(x) + f(x)V_{\epsilon}'(x)$ is a sequence of functions such that its limit at each point in the phase space yields the infinite time average of $\Phi(\cdot)$ along the trajectories passing through that point.

While this impressive feature of the sequence $V_{\epsilon}(x)$ is apparent in this simple example, such sequences of increasingly near-optimal auxiliary functions $V_{\epsilon}({\bf x})$ also exist more generally for well behaved $\frac{d\textbf{x}}{dt}=\textbf{f}(\textbf{x})$ defined by sufficiently smooth vector fields $\textbf{f}$.
If we could deduce these sequences then we could bypass the dynamics altogether to estimate and evaluate long time averages $\overline{\Phi}({\bf x})$ along all trajectories.
But, alas, as of now construction of such sequences requires explicit knowledge of the flow map---complete access to all information about all trajectories \cite{CG1993,B2018}---so this approach is essentially tautological in an operational sense.
In practice at the present time we are limited to the variational methods described in section 2 to effectively compute sequences of increasingly optimal auxiliary functions.

\section{Application to Non-Autonomous and Non-polynomial systems}
\indent The theoretical formalism and computational implementation via SDP described in Sections 1 and 2 depend very much on, respectively, the autonomous nature of the dynamics and the polynomial nature of the equations of motion.
But models in many applications involve non-autonomous, i.e., “driven” systems, and non-polynomial vector fields.
Therefore, it is useful to consider how broader classes of ODEs might be recast as autonomous polynomial systems.

Periodically forced dynamics of the form
\begin{equation}
    \frac{d\textbf{x}(t)}{dt}=\textbf{f}(\textbf{x},\cos(\omega t),\sin(\omega t))
\end{equation}
with $\textbf{x}=(x_1,x_2,....,x_d)$ are particularly interesting and ubiquitous.
The traditional way of ``autonomizing'' such systems---introduce a new coordinate $x_{d+1}= t$ and extend the system dimension from $d$ to $d+1$---has the drawbacks, however, of introducing an unbounded dependent variable while retaining non-polynomial dependence on it.

For our purposes these problems can be circumvented by introducing \textit{two} new dynamical variables satisfying the polynomial sub-system
\begin{equation}
    \begin{gathered}
        \frac{dx_{d+1}}{dt}=-\omega x_{d+2} + (1 - x_{d+1}^2 - x_{d+2}^2) x_{d+1}\\
        \frac{dx_{d+2}}{dt}=\omega x_{d+1} + (1 - x_{d+1}^2 - x_{d+2}^2) x_{d+2}.
    \end{gathered}
    \label{4.2}
\end{equation}
After a uniform-in-initial-condition exponentially decaying transient, $x_{d+1} = \cos(\omega t+\phi)$ and $x_{d+2} = \sin(\omega t+\phi)$ with arbitrary phase $\phi$.
Insofar as we're ultimately interested in extreme long time behavior among all initial data, however, the phase is irrelevant: $\phi \ne 0$ corresponds to a translation of the time origin which can be absorbed into a shift in initial conditions.

Note as well that $x_{d+1} = x_{d+2} \equiv 0$ is another solution of (\ref{4.2}) that will naturally be included the maximization or minimization of $\overline{\Phi}({\bf x}_0)$ over all initial conditions.
In practice appearance of this ``spurious'' solution may be obviated theoretically by adding appropriate multiples of $x_{d+1}^2$ and/or $x_{d+2}^2$ to $\Phi$, or computationally by implementing the S-procedure.

This approach can also be used to formulate equivalent autonomous polynomial dynamics for both quasiperiodic and substantially more complex $\frac{2 \pi}{\omega}$-periodic time dependences in the vector field.
Employing a new pair of dynamical variables like those in (\ref{4.2}) for each independent frequency allows for quasiperiodic time dependence, at least for quasiperiodicity involving only a {\it finite} number of independent frequencies.
Other $\frac{2 \pi}{\omega}$-periodic time functions can be expressed  as finite linear combinations of $\cos(n \omega t)$ and $\sin(n \omega t)$, each of which in turn is a finite polynomial combination of $\cos(\omega t)$ and $\sin(\omega t)$.
The overall order of the dynamical system necessarily increases but autonomous polynomial dynamics are still sufficient to capture the systems' dynamics.

A broad class of autonomous vector fields with trigonometric variable dependence can similarly be handled similarly \cite{Parker19}.
Consider, for example,  vector fields ${\bf f}({\bf x})$ where the components $f_1$, \dots, $f_ d$ depend polynomially on $x_j$ for $j \ne i$ and on $x_i$ via $\cos x_i$ and/or $\sin x_i$ but {\it not} on $x_i$ itself, i.e., 
\begin{equation}
f_j = f_j(x_1,\dots,x_{i-1},\cos x_i,\sin x_i,x_{i+1},\dots,x_d) \ \text{for each} \ j = 1, \dots, d.
\end{equation}
For notational simplicity let us denote the ``angular'' variable $x_i(t) = \theta(t)$ and the corresponding component of the vector field
\begin{equation}
f_i = \Omega(x_1,\dots,\dots,x_{i-1},\cos \theta,\sin \theta,x_{i-1},\dots,x_d).
\end{equation}
Then augment the system with two new variables evolving according to
\begin{equation}
    \begin{split}
        \frac{dx_{d+1}}{dt}=-\Omega \, x_{d+2} + (1 - x_{d+1}^2 - x_{d+2}^2) \, x_{d+1}\\
        \frac{dx_{d+2}}{dt}=\Omega \, x_{d+1} + (1 - x_{d+1}^2 - x_{d+2}^2) \, x_{d+2}.
    \end{split}
    \label{4.5}
\end{equation}
After transients,
\begin{equation}
\begin{split}
x_{d+1}(t) &= \cos \left(\int_0^t \Omega \, ds + \theta_0 \right) \ \text{and}\\
x_{d+2} &= \sin \left(\int_0^t \Omega \, ds + \theta_0 \right)
\end{split}
\label{CS}
\end{equation}
where $\theta_0$ is determined by initial data.

The claim now is that solutions of the original $d$-dimensional system
\begin{equation}
\frac{dx_k}{dt} = f_k(x_1,\dots,x_{i-1},\cos x_i,\sin x_i,x_{i+1},\dots,x_d) \ \text{for} \ k = 1, \dots, d
\end{equation}
are in 1-to-1 correspondence with solutions of the $(d+1)$-dimensional system consisting of (\ref{4.5}) and the remaining $d-1$ differential equations
\begin{equation}
\begin{split}
         \frac{dx_j}{dt} &= f_j(x_1,\dots,x_{i-1},x_{d+1},x_{d+2},x_{i+1},\dots,x_d) \\
         &\quad \quad \text{for} \ j = 1, \dots, i-1, i+1, \dots, d.
\end{split}
\end{equation}

The new system does not involve the original $x_i$ variable which evolves passively via $dx_i/dt = \Omega(x_1,\dots,x_{i-1},x_{d+1},x_{d+2},x_{i+1},\dots,x_d)$, the right hand side of which does not depend on $x_i$.
Variation in the arbitrary phase $\theta_0$ in (\ref{CS}) corresponds to a translation of the initial condition for the eliminated $x_i$ variable, and this does not matter when we are concerned with functions of interest $\Phi$ that only depend on $x_1,\dots,x_{i-1},x_{d+1},x_{d+2},x_{i+1},\dots,x_d$ extremized over trajectories.

We remark that the $(1 - x_{d+1}^2 - x_{d+2}^2) \, x_{d+1}$ and $(1 - x_{d+1}^2 - x_{d+2}^2) \, x_{d+1}$ terms in both (\ref{4.2}) and (\ref{4.5}), enforcing amplitude constraints, may be dropped and replaced with the S-procedure to constrain $x_{d+1}$ and $x_{d+2}$ to circles in their subspace of the phase space.
In the following subsections we illustrate these approaches and their robustness by converting the periodically forced Duffing equation and the damped-driven pendulum into autonomous polynomial form and applying the SOS/SDP technologies.

\subsection{The Periodically Driven Duffing Equation.}
The damped driven Duffing system is the non-autonomous second order nonlinear ODE
\begin{equation}
    \ddot{x}+\delta \dot{x}+\alpha x+\beta x^3=F\cos(\omega t).
    \label{Duffing}
\end{equation}
It has received widespread attention for its various engineering applications, as a simple paradigmatic model that displays dynamical hysteresis, and for exhibiting chaotic behavior \cite{Kovacic11}.

The Harmonic Balance method produces $\frac{2\pi}{\omega}$-periodic approximate solutions via insertion of ansatz
\begin{equation}
    x(t)=A\cos(\omega t)+B\sin(\omega t).
    \label{ansatz}
\end{equation}
into (\ref{Duffing}) and projecting onto $\cos(\omega t)$ and $\sin(\omega t)$.
Harmonic Balance yields an implicit prediction for the frequency response curve in the form
\begin{equation}
    \Big [\Big (\omega^2 -\alpha -\frac{3}{4}\beta R^2\Big )^2+(\delta \omega)^2\Big ]R^2-F^2=0
    \label{HO}
\end{equation}
where $R=\sqrt{A^2+B^2}$.
For fixed parameters $\alpha$, $\beta$, $F$, and $\delta$, one can solve for the roots (\ref{HO}) to deduce the oscillation amplitude $R$.

When $\alpha>0$ and $\beta>0$ or $\beta<0$ we say that the nonlinearly perturbed oscillator has been ``stiffened'' or ``softened'' and the frequency response curve tilts to the right or to the left respectively; see Figure 1. 
\begin{figure}[h!]
    \includegraphics[width=\textwidth, height=6cm]{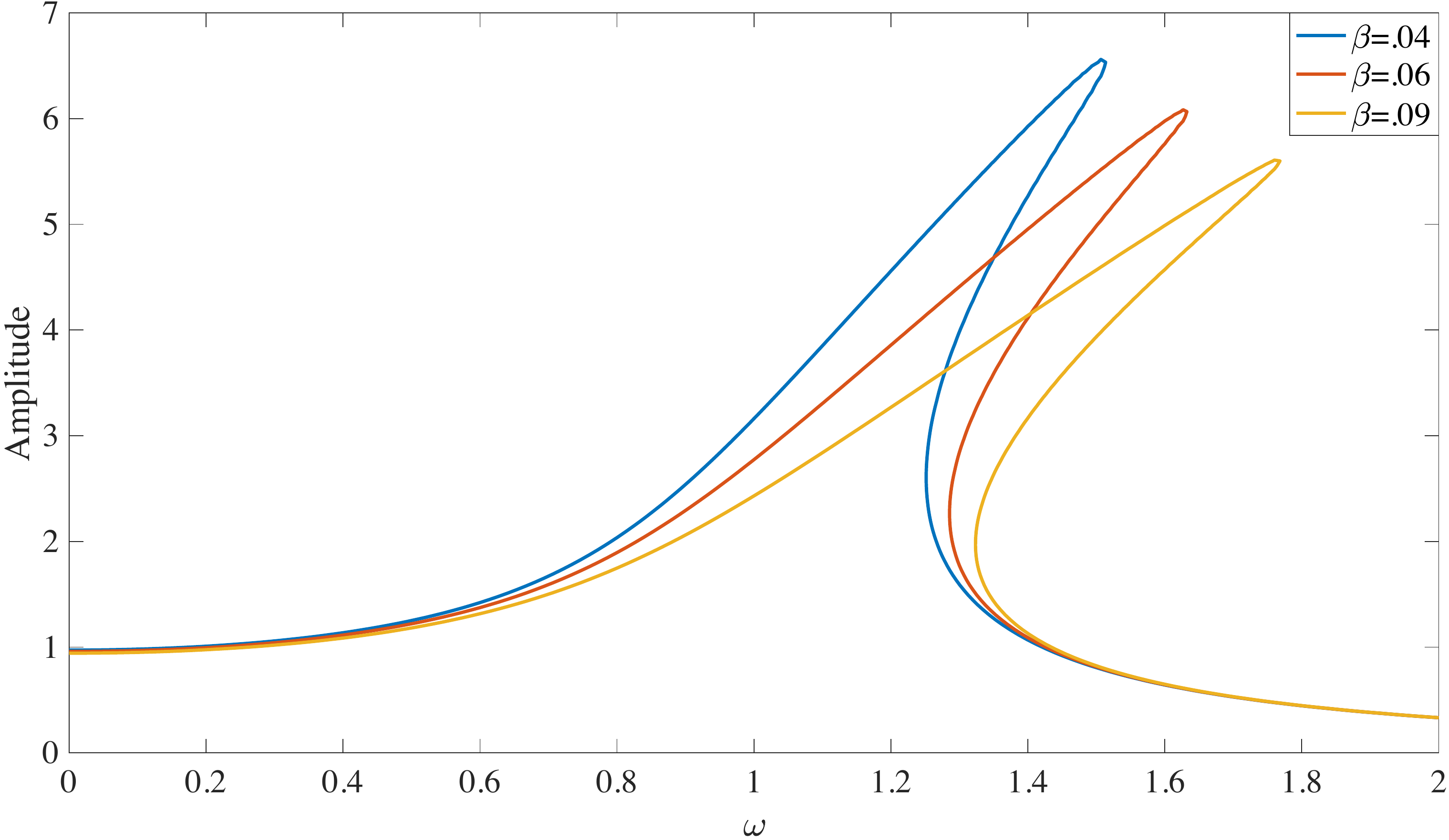}
    \caption{Harmonic Balance approximate mean amplitude $R=\sqrt{A^2+B^2}$ vs.~driving frequency $\omega$ with $\delta=.1$, $\alpha=1$, $F=1$, and $\beta=.04, .06, .09$. }
    \label{fig:Duffing}
\end{figure}

A natural question is to ask how well the Harmonic Balance method approximates true solutions of (\ref{Duffing}).
In particular, we can compare its predictions with independent approaches to recover the frequency response curves like those in Figure 1.
In the following we employ the auxiliary function method implemented in an SDP via SOS techniques.

The Duffing equation (\ref{Duffing}) is not of the form (\ref{AODE}) so we proceed by augmenting it with two additional variables to make the system autonomous.
It is then realized as the 4-dimensional first order system
\begin{equation}
\begin{split}
    &\dot{x}=y\\
    &\dot{y}=z_2-\delta y-\alpha x-\beta x^3\\
    &\dot{z_1}=\omega z_2\\
    &\dot{z_2}=-\omega z_1,
    \end{split}
\end{equation}
where the amplitudes of $z_1$ and $z_2$ will be enforced by the S-procedure so that $z_1=F\sin(\omega t+\phi)$ and $z_2=F\cos(\omega t+\phi)$ with phase $\phi$ determined by initial conditions but which is irrelevant for long time averages.
\par When the function to be maximized is
\begin{equation}
    \Phi(x,y,z_1,z_2)=x^2,
\end{equation}
the relevant SDP is 
\begin{equation}
    \begin{gathered}
    \min \text{U}\\
    \text{s.t.}\,\, \text{U}-x^2-\textbf{f}(x,y,z_1,z_2)\cdot \nabla \text{V}+ \dots \\
    \dots + \text{S}(x,y,z_1,z_2)(F^2-{z_1}^2-{z_2}^2)\in \mathcal{S}[x,y,z_1,z_2]\\
    \text{S}\in \mathcal{S}[x,y,z_1,z_2].
    \end{gathered}
\end{equation}

We now systematically increase the polynomial degrees of both V and S until sharp bounds are achieved. 
(Lower bounds on $\overline{x^2}$ can be computed by negating $\Phi$, performing the SDP, and taking the absolute value of the resulting U.)

\begin{figure}[h!]
\begin{center}
    \includegraphics[width=\textwidth, height=6.4cm]{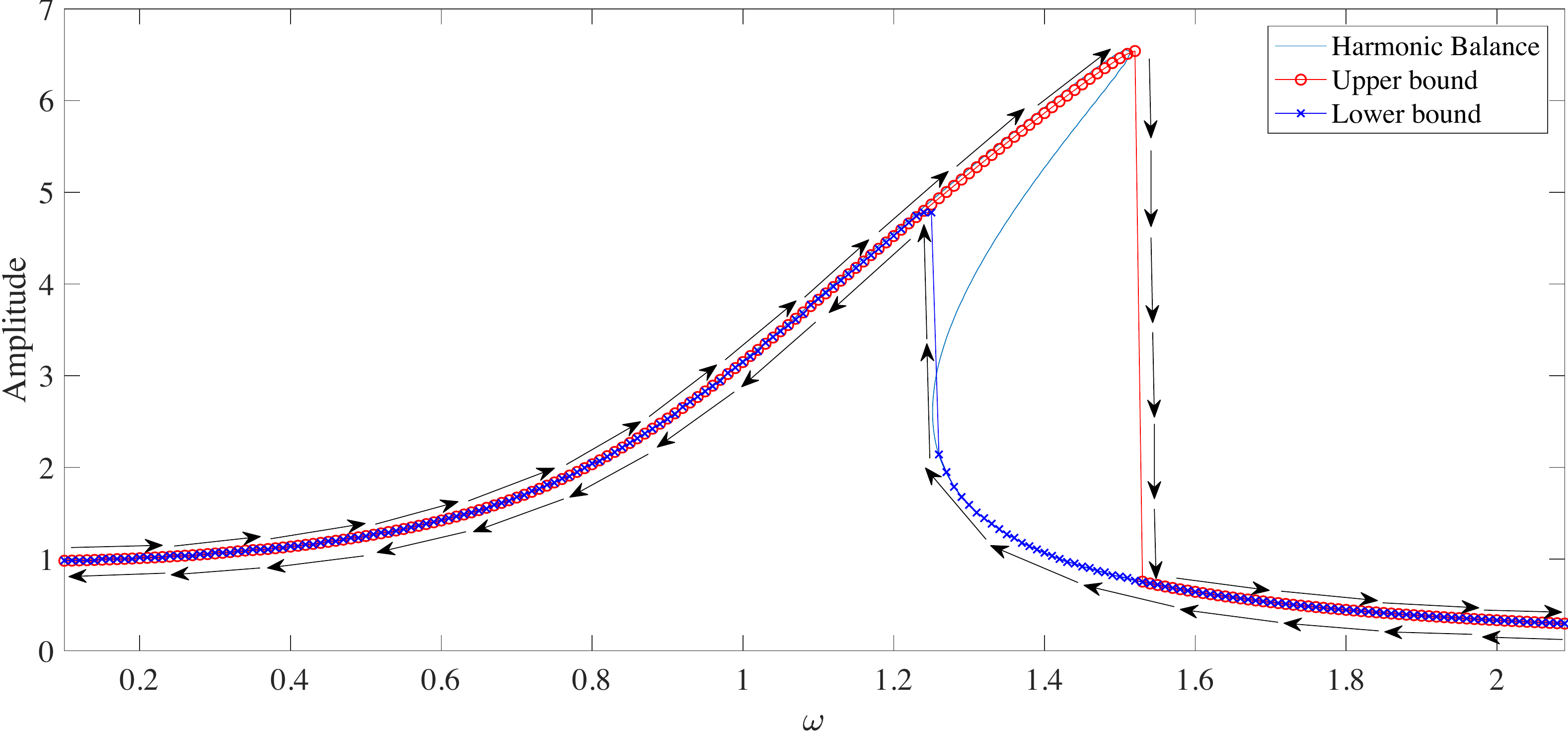}
    \vspace*{-.5cm}\caption{Harmonic Balance approximate mean amplitude, $\sqrt{2 \overline{x^2}}$, and upper and lower bounds on the solutions' mean amplitude vs.~driving frequency $\omega$ with $\delta=.1$, $\alpha=1$, $\beta=.04$, and $F=1$ for a degree 10 polynomial auxiliary function.}
    \label{fig:Duffing}
    \end{center}
\end{figure}

The Harmonic Balance approximation of the mean amplitude agrees remarkably well with the upper and lower bounds on the true solution's mean amplitude; see Figure 2.
The differences between the upper and lower bounds plotted in Figure 3 suggest that they agree (to computational precision) for points on the frequency response curve that are single valued when the degree of the auxiliary function is sufficiently high.
Not unexpectedly, there is an order 1 difference between the bounds when the curve is multi-valued.
We conclude that for this sort of small amplitude forcing and weak nonlinearity, the Harmonic Balance approximation does exceptionally well quantitatively approximating the true solution's mean amplitude---even though the forcing and nonlinearity are strong enough to induce multi-stability and hysteresis.

\begin{figure}[h!]
    \includegraphics[width=\textwidth, height=7.5cm]{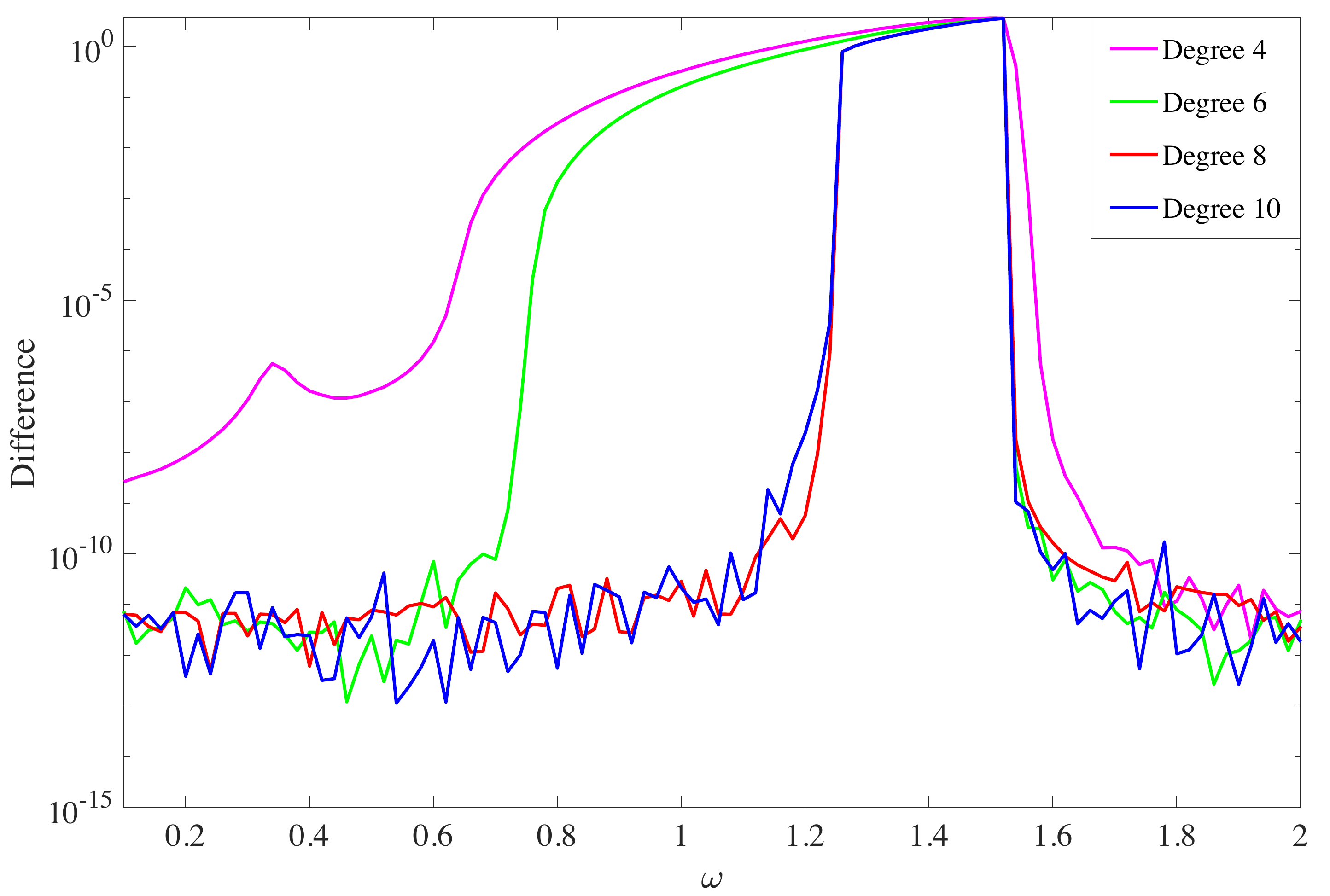}
    \vspace*{-.5cm}\caption{Difference between the upper and lower bounds on the solution's mean amplitude vs.~driving frequency for degree 4, 6, 8 and 10 polynomial auxiliary functions.}
    \label{fig:Error}
\end{figure}

It is worthwhile remarking that if we consider the degree of the auxiliary functions as a parameter then there seem to be $\omega$-dependent thresholds in the parameter space for which the degree 6 bounds become sharp.
In this example a transition occurs at $\omega \approx .7$ which, to our knowledge, is no particularly special frequency value.
It suggests that exact extreme orbits live in semi-algebraic sets, $\overline{\Phi}^* = \Phi({\bf x}) + {\bf f}({\bf x})\cdot{\bf \nabla}V({\bf x})$ for polynomial optimal $V$, whose complexity (degrees) change discontinuously with the system parameters.
Under what conditions we might expect such transitions to occur, especially with a smoothly varying parameter such as $\omega$, is an question for research in its own right.

\subsection{The Damped Periodically Driven Pendulum}
Consider a damped and periodically driven pendulum dynamics defined by the non-polynomial and non-autonomous $2^{nd}$ order ODE\begin{equation}
    \ddot{\theta}+\gamma\dot{\theta}+\sin(\theta)=F\cos(\omega t).
    \label{pend}
\end{equation}

For weak forcing, the sinusoidal non-linearity may be modeled by expanding the $\sin(\theta)$ term in a Taylor series, and we employ a procedure similar to that of the Duffing example here to test the validity of the Harmonic Balance approximation. We expect the two term expansion of the $\sin(\theta)$ term---that results in a Duffing equation---to perform poorly, however, for moderately large forcing amplitude.
Hence, we expand the $\sin(\theta)$ term in a Taylor Series to 7th order and employ (\ref{ansatz}) to obtain an approximate frequency response curve.
The result is
\begin{equation}
\begin{gathered}
    \frac{R^2(R^6+1152R^2-48R^4-9216)^2}{84934656}+R^2\omega^4\\
    +\frac{R^2(R^6+4608(\gamma^2-2)+1152R^2-48R^4)\omega^2}{4608}-F^2=0,
   \end{gathered}
\end{equation}
where $R=\sqrt{A^2+B^2}$.
The calculation is tedious and purely algebraic, but plots of $R$ vs.~$\omega$ for several forcing amplitudes are shown in Figure 4.
\begin{figure}[h!]
    \includegraphics[width=\textwidth, height=6.6cm]{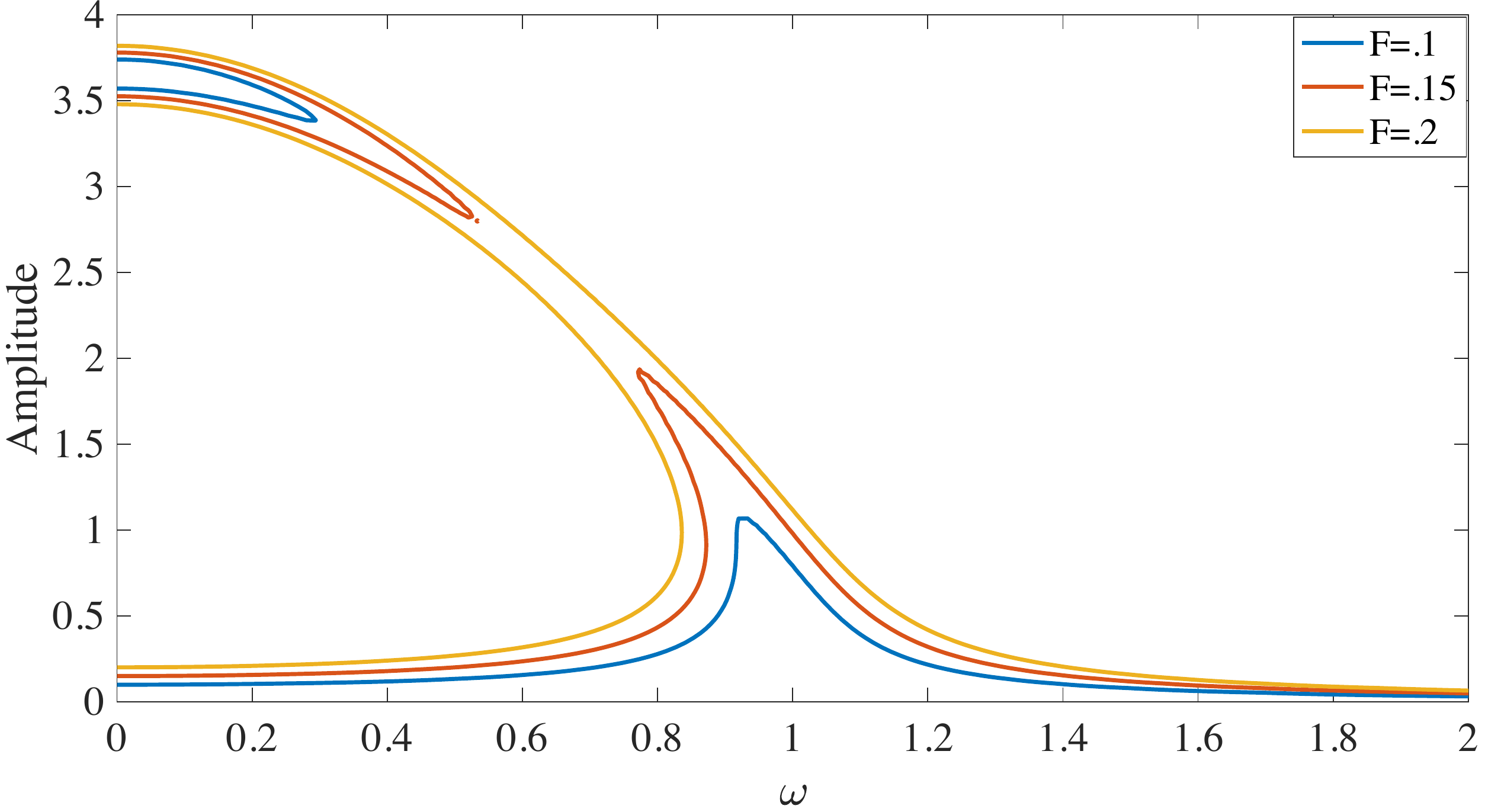}
    \vspace*{-.5cm}\caption{Plot of the Harmonic Balance approximate mean amplitude vs.~$\omega$ with $\gamma=.1$ and $F= 0.10, 0.15 \text{ and } 0.20$. }
    \label{fig:Duffing}
\end{figure}

In this example we will seek to compare the mean mechanical energy $E = \frac{1}{2}(\dot{\theta})^2-\cos(\theta)$ from the Harmonic Balance approximation with auxiliary function bounds on solutions to (\ref{pend}). 
But as written, (\ref{pend}) is neither polynomial nor autonomous which prevents immediate implementation of the polynomial optimization via an SDP.

Augmenting the system with \textit{four} additional variables, however, we may re-write equation (\ref{pend}) as the $4$-dimensional first order polynomial system
\begin{equation}
    \begin{split}
        \dot{\phi}=z_1-\gamma\phi-\psi_1\\
        \dot{\psi_1}=\phi\psi_2\\
        \dot{\psi_2}=-\phi\psi_1\\
        \dot{z_1}=\omega z_2\\
        \dot{z_2}=-\omega z_1.
    \end{split}
\end{equation}

The quantity of interest to extremize is the total energy plus $z_1^2$ given by
\begin{equation}
    E + z_1^2 = \frac{1}{2}(\dot{\theta})^2-\cos(\theta)+z_1^2=\frac{1}{2}(\phi)^2-\psi_2+{(z_1)}^2 = \Phi.
\end{equation}
The $z_1^2$ makes the SDP more computationally tractable and, because $\overline{z_1^2} = \frac{1}{2}$, we can interpret the upper-bound and lower bounds on the mean energy as a $\frac{1}{2}$ shift down and up, respectively.
Letting $\textbf{x}=[\phi,\psi_1,\psi_2,z_1,z_2]$, the semi-definite program for upper bounds becomes
\begin{equation}
    \begin{gathered}
    \min \text{U}\\
    \text{s.t.} \,\, \text{U}-\Phi(\textbf{x})-\textbf{f}(\textbf{x})\cdot \nabla \text{V}(\textbf{x})
    +C_1(\textbf{x})+C_2(\textbf{x}) \in \mathcal{S}(\textbf{x})\\
    S_1,S_2\in \mathcal{S}(\textbf{x}),
    \end{gathered}
    \label{37}
\end{equation}
where $C_1(\textbf{x})=S_1(\textbf{x})(F^2-{z_1}^2-{z_2}^2)$ and $C_2(\textbf{x})=S_2(\textbf{x})(1^2-{\psi_1}^2-{\psi_2}^2)$. 
Lower bounds on $\overline{\Phi}$ are computed just as in the Duffing setting.

\begin{figure}[h!]
    \includegraphics[width=\textwidth, height=7.7cm]{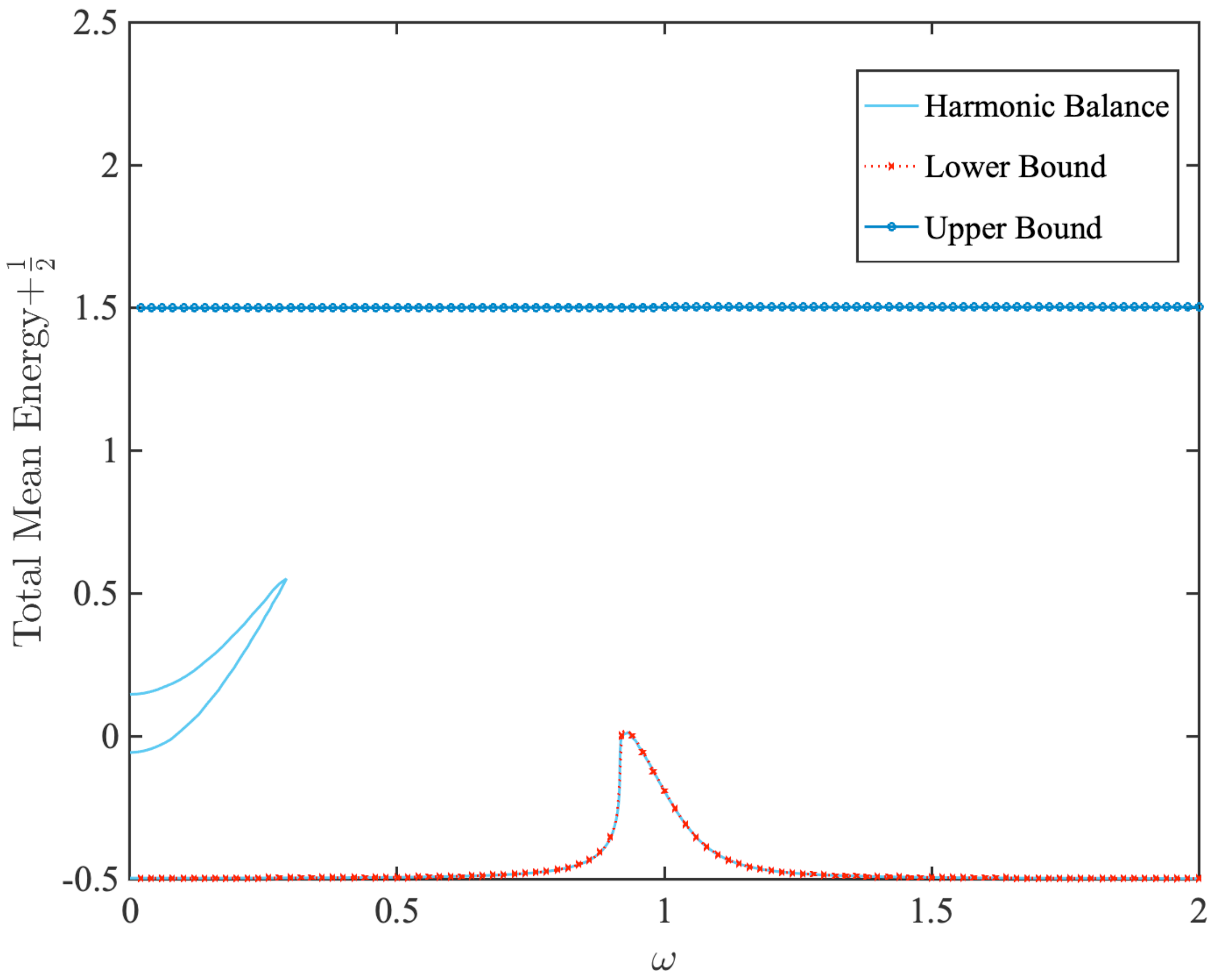}
    \vspace*{-.6cm}\caption{Plot of the bounds and Harmonic Balance approximate total mean energy vs. driving frequency, $\omega$,  with $\gamma=.1$, $F=0.10$ and degree 6 polynomial auxiliary functions.
    }
    \label{fig:Duffing}
\end{figure}

Performing the SDP in (\ref{37}), we find that the auxiliary function method's lower bound on the mean energy and the harmonic balance approximation to the mean energy can agree quite nicely---for sufficiently weak forcing.
See Figure 5.
As the forcing amplitude increases, however, the Harmonic Balance approximation is bound to fail.

On the other hand the upper bound in Figure 5 clearly does not correspond to the Harmonic Balance approximations we found.
Indeed, the upper bound with $\overline{\Phi}^* \approx 1.5$ in Figure 5 suggests that there is a solution that spends most of the time oscillating weakly around $\theta=\pi$ as illustrated in Figure 6.
Due to its dynamical instability, however, one would never expect to discover it via direct numerical simulation. 


\begin{figure}[h!]
\begin{center}
    \includegraphics[width=5cm, height=5cm]{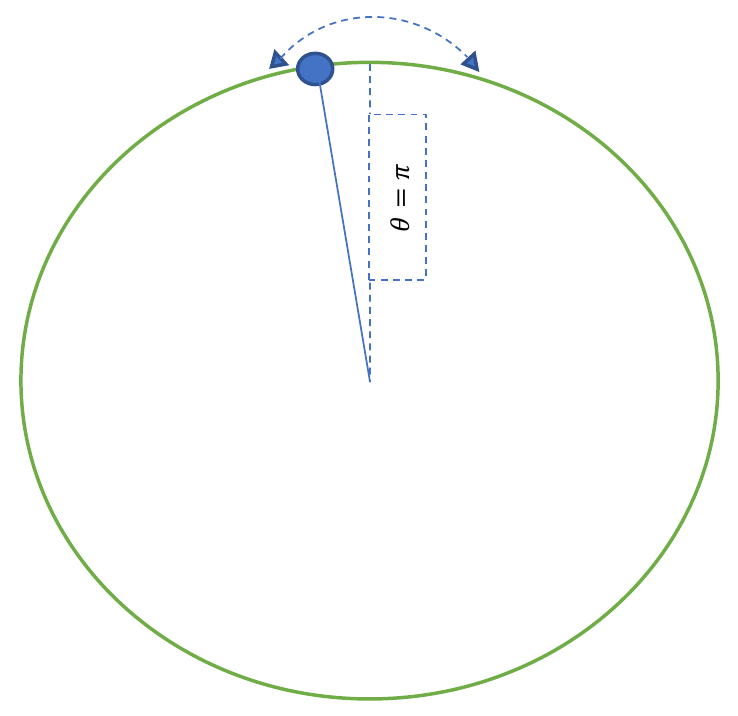}
    \caption{A potentially unstable solution ocillating about a neighborhood of $\theta=\pi$.}
    \label{fig:HP1}
    \end{center}
\end{figure}

With this interpretation in mind, we can make the linear change of variables such that $\theta'=\pi-\theta$.
Then when $\theta$ has low potential energy $\theta'$ has high potential energy and vice versa.
Figure 7 is the Harmonic Balance approximation with the Taylor expansion performed about $\theta=\pi$, the analog of Figure 4.
Meanwhile Figure 8 shows that the harmonic balance approximation of the high potential solution's total mean energy agrees quite well with the auxiliary function upper bound on the true solution's total mean energy.

\begin{figure}[h!]
\begin{center}
   \includegraphics[width=\textwidth, height=7.7cm]{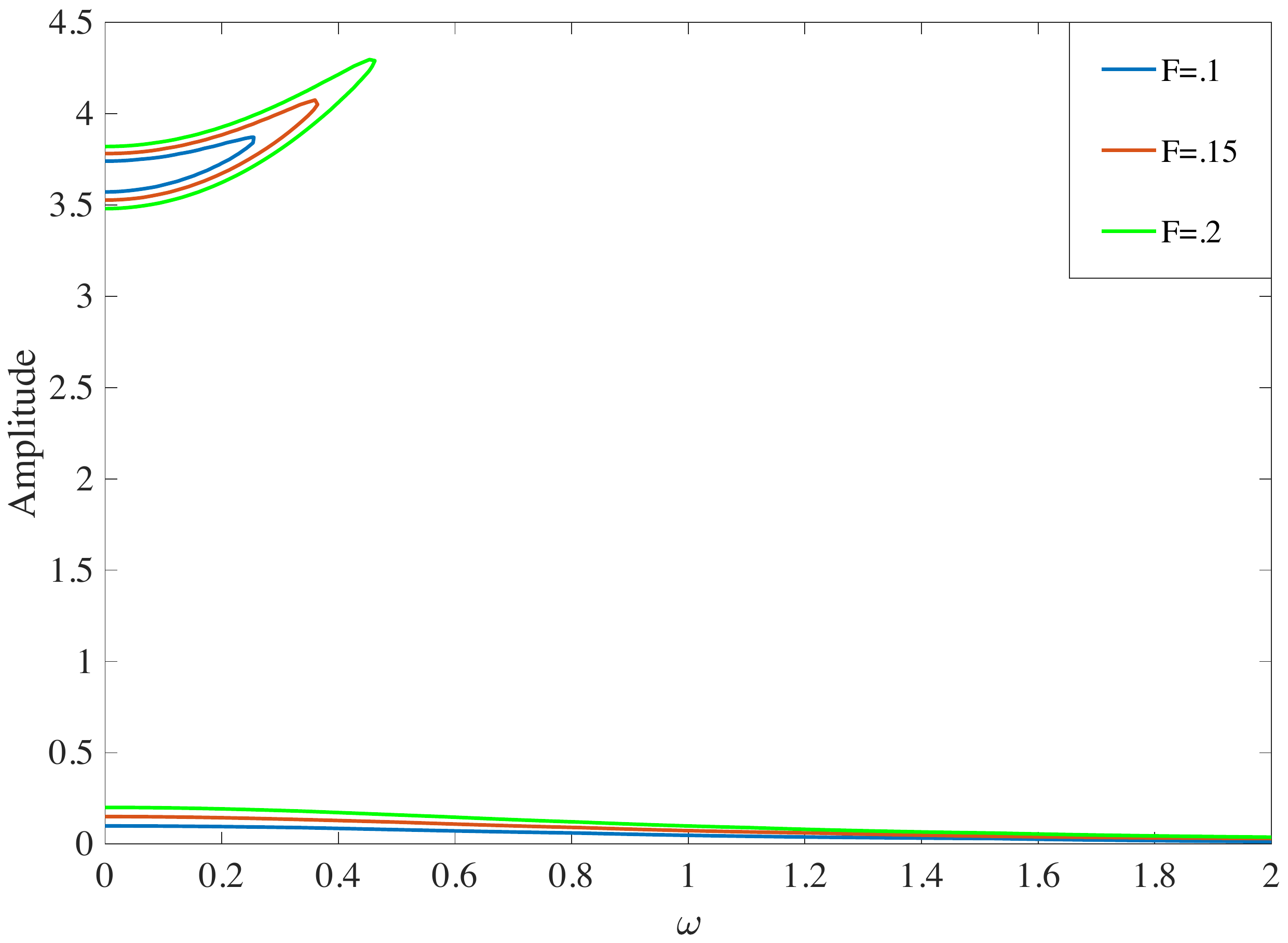}
     \vspace*{-.6cm}\caption{Harmonic Balance approximate mean amplitude about $\theta = \pi$ vs.~driving frequency $\omega$ with $\gamma=.1$ and $F$ = 0.10, 0.15 and 0.20.}
    \label{fig:HP1}
    \end{center}
\end{figure}

\begin{figure}[h!]
\vspace{.6cm}
    \includegraphics[width=\textwidth, height=7.7cm]{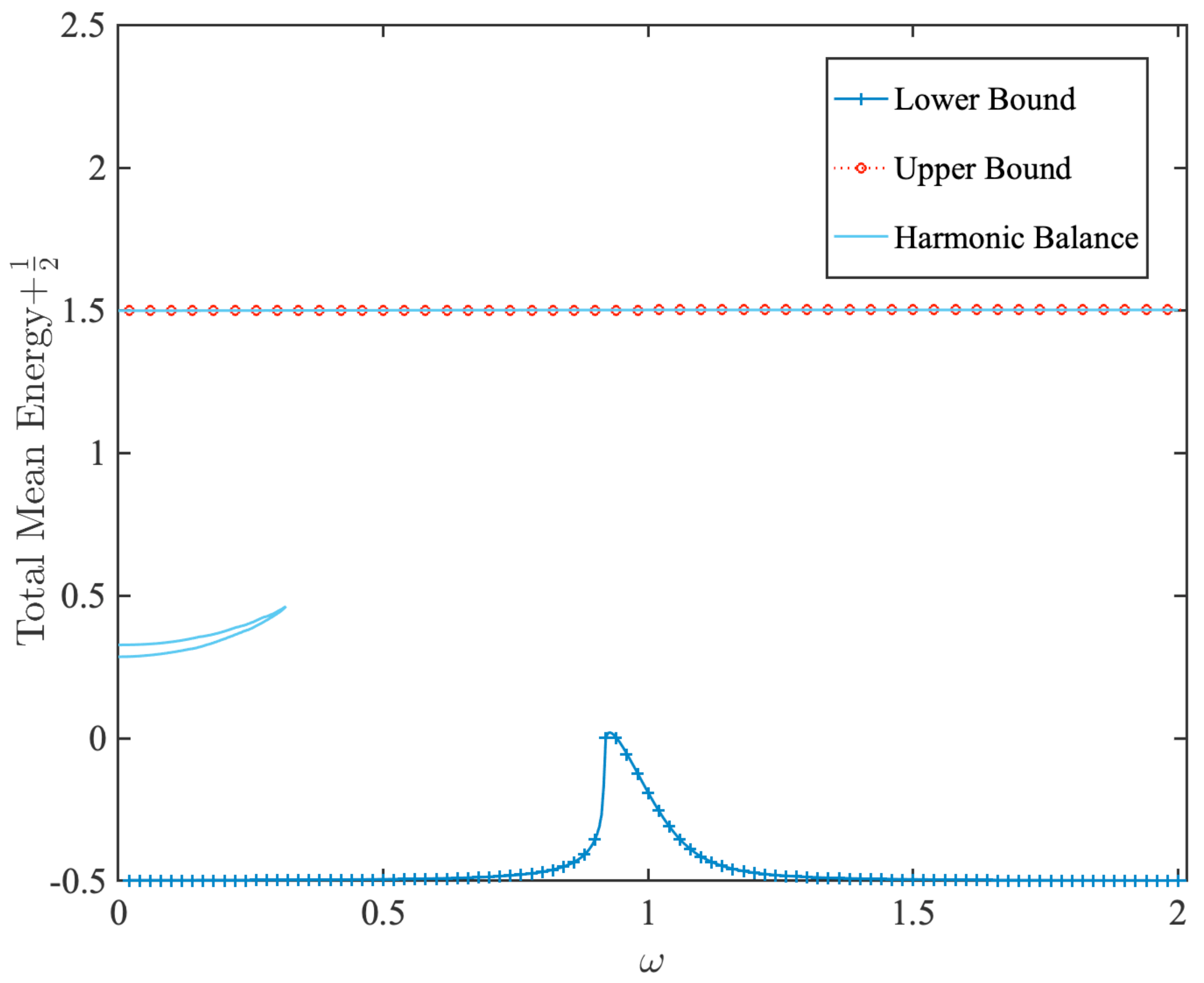}
    \vspace*{-.6cm}\caption{Plot of the bounds and Harmonic Balance approximation (about $\theta = \pi$) of $\overline{\Phi}$ vs.~driving frequency $\omega$ with $\gamma=.1$, $F=0.20$ and degree 6 polynomial auxiliary functions.}
    \label{fig:HP1}
\end{figure}

This example illustrates one of the operational ``quirks'' of the auxiliary function method: it produces upper bounds or lower bounds on the chosen $\Phi$ across all potential initial conditions including those that breed not readily observed unstable solutions.
Of course the knowledge of the existence of such unstable solutions is frequently a concern---it is certainly the central concern for control-of-chaos applications---but if one is interested in estimates of long time averages of $\Phi$ on particular solutions (or branches of solutions) there is currently no supplementary procedure that one can employ to ensure that the bounds computed correspond with specific trajectories. 


\section{Summary \& Discussion}

There are several key remarks to be made regarding both what we've done and future directions.
First is that we've displayed the robustness of the auxiliary function and SOS-SDP technologies to handle both non-autonomous and (certain forms of) trigonometric dependence in non-linear ODEs.
Such systems are ubiquitious in applications and as canonical case studies.
Second, we note that the SOS technology is computationally tractable even for much larger ODE systems that we considered here.
The plots for this paper were produced on a standard laptop, and the procedure of augmenting a dynamical system with additional polynomial degrees of freedom appears robust. 
Additionally, we observe that sharp---within computer precision---bounds are often recovered for polynomial auxiliary functions of reasonably restricted degree.
This appears to be the case not only in this work, but also various others \cite{Chern14,Fantuzzi16}, so the SDP algorithm is able to concentrate the relative coefficients on potentially severely truncated polynomials for which sharp bounds are guaranteed. 

We can clearly see areas of ongoing research for which this technology is broadly applicable.
Many such problems can directly be cast in the light of the auxiliary function method and have been inaccessible until due to both theoretical and computational limitations. 
There are a variety applied science and engineering application where moderately low dimensional ODE systems serve as central models for both conceptual and design purposes.
These include energy harvesting \cite{Mann09,Wei} where the challenge is to optimally extract power from vibrations of a continuously stimulated mechanical body where mathematical models often consist of periodically driven nonlinear oscillators \cite{Erturk}.
Another area is the periodic operation of chemical and biochemical reactors \cite{Silveston} where the task is to optimize the time-average production of certain byproducts.
Mass action and related kinetic models often consist of ODEs with polynomial vector fields.
Circadian \cite{Gonze,Komin} or seasonally forced \cite{Greenman,Taylor} models in biology, ecology and epidemiology are often described by such periodically driven ODEs with polynomial vector fields as well.

Finally, we recognize the frontier for application of the auxiliary function approach and related numerical methods to systems described by partial differential equations (PDEs).
Of course PDEs are often approximated by finite---albeit sometimes very large---systems of ODEs, but fundamental mathematical and computational questions remain for future research.

\section{Acknowledgements}
\par Some of the ideas presented here were developed in discussions with David Goluskin (Victoria) and Jeremy Parker (Cambridge) supported in part by US National Science Foundation Awards DMS-1813003 at the University of Michigan and OCE-1829864 for the Geophysical Fluid Dynamics Program at Woods Hole Oceanographic Institution.

\bibliographystyle{unsrt}

\end{document}